\documentclass[11pt]{amsart}
\usepackage{fullpage,amssymb,amsmath,hyperref,color,enumitem,caption,comment}
\usepackage[noadjust]{cite}
\usepackage[abs]{overpic}
\pdfoutput=1

\newcommand{\bbA}{\mathbb{A}}
\newcommand{\bbC}{\mathbb{C}}

\newcommand{\bbN}{\mathbb{N}}
\newcommand{\bbP}{\mathbb{P}}
\newcommand{\bbQ}{\mathbb{Q}}

\newcommand{\bbZ}{\mathbb{Z}}

\newcommand{\calA}{\mathcal{A}}
\newcommand{\calB}{\mathcal{B}}

\newcommand{\calK}{\mathcal{K}}
\newcommand{\calL}{\mathcal{L}}
\newcommand{\calM}{\mathcal{M}}
\newcommand{\calO}{\mathcal{O}}

\newcommand{\calR}{\mathcal{R}}

\newcommand{\calX}{\mathcal{X}}
\newcommand{\calZ}{\mathcal{Z}}

\newcommand{\calKbar}{\overline{\calK}}

\renewcommand{\hbar}{\overline{h}}

\newcommand{\frakp}{\mathfrak{p}}
\newcommand{\frakq}{\mathfrak{q}}
\newcommand{\hhat}{\widehat{h}}

\newcommand{\Spec}{\operatorname{Spec}}

\newcommand{\Res}{\operatorname{Res}}

\newcommand{\disc}{\operatorname{disc}}

\renewcommand{\bar}{\overline}
\renewcommand{\tilde}{\widetilde}
\renewcommand{\phi}{\varphi}
\renewcommand{\epsilon}{\varepsilon}

\newtheorem{thm}{Theorem}[section]
\newtheorem{lem}[thm]{Lemma}
\newtheorem{prop}[thm]{Proposition}
\newtheorem{cor}[thm]{Corollary}

\theoremstyle{definition}
\newtheorem{ques}[thm]{Question}
\newtheorem{rem}[thm]{Remark}
\theoremstyle{remark}

\numberwithin{equation}{section}

\title{Preperiodic portraits for unicritical polynomials over a rational function field}
\author{John R. Doyle}
\address{Department of Mathematics \\
University of Rochester \\
Rochester, NY 14627} 
\email{john.doyle@rochester.edu}

\begin{document}

\begin{abstract}
Let $K$ be an algebraically closed field of characteristic zero, and let $\calK := K(t)$ be the rational function field over $K$. For each $d \ge 2$, we consider the unicritical polynomial $f_d(z) := z^d + t \in \calK[z]$, and we ask the following question: If we fix $\alpha \in \calK$ and integers $M \ge 0$, $N \ge 1$, and $d \ge 2$, does there exist a place $\frakp \in \Spec K[t]$ such that, \emph{modulo $\frakp$}, the point $\alpha$ enters into an $N$-cycle after precisely $M$ steps under iteration by $f_d$? We answer this question completely, concluding that the answer is generally affirmative and explicitly giving all counterexamples. This extends previous work by the author in the case that $\alpha$ is a constant point.
\end{abstract}



\maketitle

\section{Introduction}\label{sec:intro}
Let $F$ be a field, and let $\phi(z) \in F(z)$ be a rational function, thought of as a self-map of $\bbP^1(F)$. For an integer $n \ge 0$, we denote by $\phi^n$ the $n$-fold composition of $\phi$; that is, $\phi^0$  is the identity map, and $\phi^n = \phi \circ \phi^{n-1}$ for each $n \ge 1$. We say that $\alpha \in \bbP^1(F)$ is \textbf{periodic} for $\phi$ if there exists an integer $N \ge 1$ for which $\phi^N(\alpha) = \alpha$; the minimal such $N$ is called the \textbf{period} of $\alpha$. More generally, we say that $\alpha$ is \textbf{preperiodic} if there exist integers $M \ge 0$ and $N \ge 1$ such that $f^M(\alpha)$ has period $N$; if $M$ is minimal, we say that $(M,N)$ is the \textbf{preperiodic portrait} (or simply \textbf{portrait}) of $\alpha$ under $\phi$. If $M \ge 1$, then we say that $\alpha$ is \textbf{strictly preperiodic}. The \textbf{orbit} of $\alpha$ under $\phi$ is the set
	\[
		\calO_\phi(\alpha) := \{\phi^n(\alpha) : n \in \bbZ_{\ge 0}\}.
	\]
Note that $\alpha$ is preperiodic for $\phi$ if and only if $\calO_\phi(\alpha)$ is finite. We say that a point is \textbf{wandering} if it is not preperiodic.

Let $\calM_F$ denote the set of places of $F$. (If $F$ is a function field, we require the places to be trivial on the constant subfield.) For a place $\frakp \in \calM_F$, let $k_\frakp$ denote the residue field at $\frakp$. Given a rational map $\phi$ and a place $\frakp$, one can consider the reduction of $\phi$ at $\frakp$: Write $\phi(z) = p(z)/q(z)$ with coprime $p,q \in F[z]$, normalized so that all coefficients are integral at $\frakp$, and at least one coefficient is a unit at $\frakp$. Then the \textbf{reduction} of $\phi$ at $\frakp$ is the map $\tilde{\phi}(z) \in k_\frakp(z)$ obtained by reducing the coefficients modulo $\frakp$. We say that $\phi$ has \textbf{good reduction} modulo $\frakp$ if $\deg \tilde{\phi} = \deg \phi$. If $\frakp$ is a place of good reduction for $\phi$, then we say that a point $\alpha \in \bbP^1(F)$ is \textbf{preperiodic for $\phi$ modulo $\frakp$} if the reduction $\tilde{\alpha} \in \bbP^1(k_\frakp)$ is preperiodic for the map $\tilde{\phi}$. We say that $\alpha$ has \textbf{(preperiodic) portrait $(M,N)$ for $\phi$ modulo $\frakp$} if $\tilde{\alpha}$ has preperiodic portrait $(M,N)$ for $\tilde{\phi}$.

If $\alpha$ is not preperiodic for $\phi$, it may still be true that $\alpha$ is preperiodic for $\phi$ modulo $\frakp$ at some place $\frakp$ of good reduction. For example, this will necessarily be true if $F$ has finite residue fields. We now consider the following more specific question regarding preperiodicity modulo places of $F$:
\begin{ques}\label{ques:main}
Let $\phi \in F(z)$. Fix $\alpha \in F$ and integers $M \ge 0$, $N \ge 1$. Does there exist a place $\frakp \in \calM_F$ of good reduction for $\phi$ such that $\alpha$ has preperiodic portrait $(M,N)$ for $\phi$ modulo $\frakp$?
\end{ques}
If the answer to Question~\ref{ques:main} is ``yes," we will say that $\alpha$ \textbf{realizes portrait $(M,N)$ for $\phi$}.

Question~\ref{ques:main} has been studied by multiple authors in the case that $F$ is a number field, dating back to related questions addressed by Bang \cite{bang:1886} and Zsigmondy \cite{zsigmondy:1892} in the late nineteenth century. Much more recently, Ingram and Silverman \cite{ingram/silverman:2009} conjectured that if $F$ is a number field and $\alpha \in F$ is a wandering point for $\phi$, then $\alpha$ realizes all but finitely many portraits for $\phi$. Faber and Granville \cite{faber/granville:2011} later gave counterexamples to this conjecture, noting that if $\phi(z) \in \bbQ(z)$ is totally ramified over all points of period $N$, then a given $\alpha \in \bbQ$ will \emph{fail} to realize portrait $(M,N)$ for all but finitely many $M$. Ghioca, Nguyen, and Tucker \cite{ghioca/nguyen/tucker:2015} subsequently pointed out that if $\phi$ is totally ramified over $\phi^M(\alpha)$ for some $M \ge 1$, then $\alpha$ cannot realize portrait $(M,N)$ for any $N \in \bbN$; their main result (\cite[Thm. 1.3]{ghioca/nguyen/tucker:2015}) is that these are the only obstructions to the analogue of the Ingram-Silverman conjecture in the setting where $F$ is the function field of a curve over an algebraically closed field of characteristic zero. They also claim that the appropriate modification of the Ingram-Silverman conjecture over number fields may be proven, under the assumption of the $abc$-conjecture, by adapting the methods of \cite{ghioca/nguyen/tucker:2015}. 

The purpose of the present article is to explicitly describe all exceptions to the result of Ghioca-Nguyen-Tucker in a natural special case. For the remainder of the paper, $K$ will be an algebraically closed field of characteristic zero, and $\calK = K(t)$ will be the rational function field over $K$. Places of $\calK$ correspond naturally to points on $\bbP^1(K)$, and the residue field at each place is isomorphic to $K$. For a point $c \in \bbP^1(K)$, we denote by $\frakp_c \in \calM_\calK$ the place corresponding to $c$.

We take our rational maps to be the \emph{unicritical polynomials}
	\[
		f_d(z) := z^d + t \in \calK[z]
	\]
of degree $d \ge 2$, which have good reduction away from $\frakp_\infty$. For each $c \in K$, we denote by $f_{d,c}$ the specialization of $f_d$ at $\frakp_c$; that is, $f_{d,c}(z) = z^d + c \in K[z]$.

Our main result fully answers Question~\ref{ques:main} with $F = \calK$ and $\phi = f_d$. For what follows, let
	\[
		\phi_1(z) := -\dfrac{t(z + 1)}{z - (t - 1)} \ \mbox{ and } \ \phi_2(z) := \dfrac{(t+1)(z - 1)}{z + t}.
	\]
\begin{thm}\label{thm:main}
Let $K$ be an algebraically closed field of characteristic zero, let $\calK := K(t)$ be the rational function field over $K$, and let $(\alpha,M,N,d) \in \calK \times \bbZ^3$ with $M \ge 0$, $N \ge 1$, and $d \ge 2$. Then there exists a place $\frakp \in \calM_\calK \setminus\{\frakp_\infty\} = \Spec K[t]$ such that $\alpha$ has preperiodic portrait $(M,N)$ under $f_d$ modulo $\frakp$ if and only if $(\alpha,M,N,d)$ does not satisfy one of the following conditions:
	\begin{itemize}
	\item $M = 1$ and $\alpha = 0$;
	\item $(M,N,d) = (0,2,2)$ and $\alpha = -1/2$;
	\item $(M,N,d) = (1,1,2)$ and $\alpha \in \calO_{\phi_1}(0) \cup \calO_{\phi_1}(\infty)$;
	\item $(M,N,d) = (1,2,2)$ and $\alpha \in \calO_{\phi_2}(0) \cup \calO_{\phi_2}(1/2) \cup \calO_{\phi_2}(\infty)$; or
	\item $(M,N,d) = (2,2,2)$ and $\alpha = \pm 1$.
	\end{itemize}
\end{thm}

\begin{rem}
The families of counterexamples in the $(1,1,2)$ and $(1,2,2)$ cases were discovered experimentally, and it was unclear whether some dynamical properties of the maps $\phi_1$ and $\phi_2$ could explain their appearance. Tom Tucker later pointed out that $\phi_1$ (resp., $\phi_2$) fixes each of the points in $\calKbar$ of portrait $(1,1)$ (resp., $(1,2)$) for $f_2$ and preserves vanishing at $\frakp_0$ (resp. $\frakp_{-1}$), the unique place at which the totally ramified point $0$ is fixed (resp., has period two) for $f_{2,c}$. Finally, we note that these exceptions have arbitrarily large height: for each $k \ge 0$, the points $\phi_1^k(0)$ and $\phi_1^k(\infty)$ have height $k$, as do the points $\phi_2^k(0)$, $\phi_2^k(1/2)$, and $\phi_2^k(\infty)$ --- see Propositions~\ref{prop:maind2N2} and \ref{prop:maind2N1}.
\end{rem}
As mentioned above, Ghioca, Nguyen, and Tucker consider Question~\ref{ques:main} for general rational maps and for function fields of arbitrary curves. Their main result \cite[Thm. 1.3]{ghioca/nguyen/tucker:2015}, when applied to the case of unicritical polynomials over $\calK$, says the following: For a fixed $d \ge 2$, if $(\alpha,M) \ne (0,1)$, and if $(M,N)$ avoids an effectively computable finite subset $\calZ(d) \subseteq \bbZ_{\ge 0} \times \bbN$, then every $\alpha \in \calK$ realizes portrait $(M,N)$ for $f_d$. Theorem~\ref{thm:main} implies that
	\[
	\calZ(d) =
		\begin{cases}
		\{(0,2),(1,1),(1,2),(2,2)\}, &\mbox{ if } d = 2;\\
		\emptyset, &\mbox{ if } d \ge 3.
		\end{cases}
	\]
Moreover, for each $(M,N) \in \calZ(2)$, Theorem~\ref{thm:main} explicitly gives all $\alpha \in \calK$ which do not realize portrait $(M,N)$ for $f_2$.

Theorem~\ref{thm:main} may also be viewed as a natural extension of a previous result of the author:

\begin{thm}[{\cite[Thm. 1.3]{doyle:portraits}}]\label{thm:old}
Let $K$ be as before, and let $(\alpha,M,N,d) \in K \times \bbZ^3$ with $M \ge 0$, $N \ge 1$, and $d \ge 2$. There exists $c \in K$ for which $\alpha$ has portrait $(M,N)$ under $f_{d,c}$ if and only if
	\[
	(\alpha,M) \ne (0,1) \mbox{ and } (\alpha,M,N,d) \not \in \left\{\left(-\frac{1}{2},0,2,2\right), \left(\frac{1}{2},1,2,2\right), \left(\pm 1, 2,2,2\right)\right\}.
	\]
\end{thm}

This is precisely the case of Theorem~\ref{thm:main} in which $\alpha$ lies in the constant subfield $K$. The proof of Theorem~\ref{thm:old} almost exclusively used the geometry of certain dynamical modular curves associated to the maps $f_d$, whereas the proof of Theorem~\ref{thm:main} requires Diophantine methods much like those used in \cite{ghioca/nguyen/tucker:2015}. In particular, the argument for the $d \ge 3$ case of Theorem~\ref{thm:main} provides a completely different proof of the $d \ge 3$ case of Theorem~\ref{thm:old} --- except for the case $M = 0$, for which we simply refer to Theorem~\ref{thm:old} for constant points. The same is true for $d = 2$, except for the cases where $M = 1$ and $N \le 3$, where the Diophantine methods are insufficient for constant points.

We now give a brief overview of the article. In \textsection \ref{sec:prelim}, we collect the main tools required for the proof of the main theorem. In \textsection \ref{sec:reduction}, we prove the $M = 0$ case of Theorem~\ref{thm:main}, and we then show that the problem for $M \ge 1$ may essentially be reduced to $M = 1$.

We prove the general case ($d \ge 3$) of Theorem~\ref{thm:main} in \textsection \ref{sec:deg3}. Focusing on the situation with $M = 1$, we apply the $abc$-theorem for function fields due to Mason and Stothers \cite{mason:1984, stothers:1981} to get a lower bound on the number of places at which $f(\alpha)$ and $f^{N+1}(\alpha)$ agree; we then show that this bound must be greater than the number of places at which either $\alpha$ is periodic or $f(\alpha)$ has period strictly less than $N$, so there must be some place at which $\alpha$ has portrait $(1,N)$. The arguments in this case are quite similar to those used in \cite{ghioca/nguyen/tucker:2015}, though we make modifications based on the specific nature of our maps $f_d$ in order to obtain sufficiently nice bounds.

The case $d = 2$ must be handled separately; this case is discussed in \textsection \ref{sec:deg2}. A technique similar to that used for $d \ge 3$ is used when $N \ge 4$. While this particular method is insufficient when $N = 3$, we are able to prove the result in this case by applying the $abc$-theorem together with properties of the period-three \emph{dynatomic polynomial} associated to the map $z^2 + t$. Unfortunately, the $abc$-theorem can no longer be applied when $N = 1$ and $N = 2$, so we handle these cases with completely different techniques, again appealing to the appropriate dynatomic polynomials. Theorem~\ref{thm:main} is then proven by combining Proposition~\ref{prop:main3} (for the case $d \ge 3$) with Propositions~\ref{prop:maind2N4}, \ref{prop:maind2N3}, \ref{prop:maind2N2}, and \ref{prop:maind2N1} (for $d = 2$ and $N \ge 4$, $N = 3$, $N = 2$, and $N = 1$, respectively).

\subsection*{Acknowledgments}
I would like to thank Tom Tucker for introducing me to this problem, as well as for a number of very helpful discussions over the course of writing this article.

\section{Preliminaries}\label{sec:prelim}

\subsection{Valuations and heights}
Let $\calL$ be a finite extension of $\calK$, which corresponds to a finite morphism of curves $\calX_\calL \to \calX_\calK \cong \bbP^1_K$. For a place $\frakp \in \calM_\calK$, we denote by $\calM_{\calL,\frakp}$ the set of places of $\calL$ that restrict to $\frakp$. Associated to each place $\frakq \in \calM_\calL$ is a valuation $v_\frakq$ and its corresponding absolute value $|\cdot|_\frakq = e^{-v_\frakq( \cdot )}$. When $\calL = \calK$, so that places correspond to points on $\bbP^1(K)$, we abuse notation and write $v_c$, $| \cdot |_c$, and $\calM_{\calL,c}$ for $v_{\frakp_c}$, $|\cdot|_{\frakp_c}$, and $\calM_{\calL,\frakp_c}$, respectively.

We normalize the valuations on $\calL$ so that $v_\frakq(\calL^\times) = \bbZ$; equivalently, if $\pi_\frakq$ is a uniformizer at $\frakq$, then $v(\pi_\frakq) = 1$. Thus, if $\frakp$ is the restriction of $\frakq$ to $\calK$, and if $\alpha \in \calK$, then $v_\frakq(\alpha) = e_{\frakq/\frakp} \cdot v_\frakp(\alpha)$, where $e_{\frakq/\frakp}$ is the ramification degree of $\frakq$ over $\frakp$. This normalization of the valuations also ensures that the product formula holds: For all $\alpha \in \calL^\times$, we have
	\[
	\prod_{\frakq \in \calM_\calL} |\alpha|_\frakq = 1,
	\mbox{ or equivalently, }
	\sum_{\frakq \in \calM_\calL} v_\frakq(\alpha) = 0.
	\]	
For each $\alpha \in \calL$, set
	\[
		h_\calL(\alpha) = -\sum_{\frakq \in \calM_\calL} \min\{v_\frakq(\alpha),0\}
			= -\sum_{\substack{\frakq \in \calM_\calL\\v_\frakq(\alpha) < 0}} v_\frakq(\alpha).
	\]
By the product formula, this is equivalent (when $\alpha \ne 0$) to defining
	\[
		h_\calL(\alpha) = \sum_{\frakq \in \calM_\calL} \max\{v_\frakq(\alpha),0\} = \sum_{\substack{\frakq \in \calM_\calL\\v_\frakq(\alpha) > 0}} v_\frakq(\alpha).
	\]
If we consider $\alpha \in \calL$ as a rational map $\calX_\calL \to \bbP^1$, then $h_\calL(\alpha)$ is simply the degree of the map. If $\calL'$ is a finite extension of $\calL$, then $h_{\calL'}(\alpha) = [\calL':\calL]h_\calL(\alpha)$ for all $\alpha \in \calL$. This allows us to give a well-defined \textbf{(absolute) height function} on all of $\calKbar$, given by
	\[
		h(\alpha) := \frac{1}{[\calL : \calK]} \cdot h_\calL(\alpha)
	\]
for any finite extension $\calL/\calK$ containing $\alpha$. Given a rational map $\phi(z) \in \calK(z)$ of degree $d \ge 2$, we also define the \textbf{canonical height} associated to $\phi$:
	\[
		\hhat_\phi(\alpha) := \lim_{n \to \infty} \frac{1}{d^n} h(\phi^n(\alpha)).
	\]
That this is well-defined follows from the fact that $h(\phi(\alpha)) = dh(\alpha) + O(1)$, where the implied constant depends only on $\phi$; see \cite[\textsection 3.2]{silverman:2007}. Note that $\hhat_\phi(\phi(\alpha)) = d\hhat_\phi(\alpha)$ for all $\alpha \in \calKbar$.

We now record a basic height identity for elements of the orbit of a point $\alpha \in \calK$.

\begin{lem}\label{lem:hhat_ineq}
Let $\alpha \in \calK$, and let $d \ge 2$. Then for each $n \ge 1$, the poles of $f_d^n(\alpha)$ are precisely $\frakp_\infty$ and the poles of $\alpha$. Moreover,
	\begin{enumerate}
	\item if $\frakp$ is a finite pole of $\alpha$, then $v_\frakp(f_d^n(\alpha)) = d^nv_\frakp(\alpha)$;
	\item $v_\infty(f_d^n(\alpha)) = d^n \cdot
		\begin{cases}
		v_\infty(\alpha), &\mbox{ if } v_\infty(\alpha) < 0;\\
		-1/d, &\mbox{ if } v_\infty(\alpha) \ge 0.
		\end{cases}$
	\end{enumerate}
Therefore $h(f_d^n(\alpha)) = d^n \cdot 
	\begin{cases}
	h(\alpha), &\mbox{ if } v_\infty(\alpha) < 0;\\
	h(\alpha) + 1/d, &\mbox{ if } v_\infty(\alpha) \ge 0.
	\end{cases}$
\end{lem}

\begin{proof}
Since $f_d^n(z)$ is a polynomial in $z$ and $t$, every pole of $f_d^n(\alpha)$ must be equal to $\frakp_\infty$ or a pole of $\alpha$. That the poles of $f_d^n(\alpha)$ are \emph{precisely} $\frakp_\infty$ and the poles of $\alpha$ then follows from parts (A) and (B), which we now prove by induction on $n$.

For $n = 1$, we have $f_d(\alpha) = \alpha^d + t$, so part (A) follows immediately from the ultrametric inequality. Furthermore, since $v_\infty(\alpha^d) \ne -1 = v_\infty(t)$, we have $v_\infty(f_d(\alpha)) = \min\{dv_\infty(\alpha),-1\}$.

Now suppose $n \ge 2$. First, let $\frakp$ be a finite pole of $\alpha$. By the induction hypothesis, $\frakp$ is a pole of $f_d^{n-1}(\alpha)$ of order $d^{n-1}v_\frakp(\alpha)$; applying the $n = 1$ case with $\alpha$ replaced by $f_d^{n-1}(\alpha)$ yields (A). We now consider $\frakp = \frakp_\infty$, in which case the induction hypothesis tells us that
	\[
		v_\infty(f_d^{n-1}(\alpha)) = d^{n-1} \cdot
			\begin{cases}
			v_\infty(\alpha), &\mbox{ if } v_\infty(\alpha) < 0\\
			-1/d, &\mbox{ if } v_\infty(\alpha) \ge 0.
			\end{cases}
	\]
Since this quantity is necessarily negative, the $n = 1$ case implies $v_\infty(f_d^n(\alpha)) = d v_\infty(f_d^{n-1}(\alpha))$, which gives us (B).

Finally, we note that for all $n \ge 1$,
	\begin{align*}
	h(f_d^n(\alpha)) = -\sum_{v_\frakp(f_d^n(\alpha)) < 0} v_\frakp(f_d^n(\alpha))
		&= -\sum_{\substack{v_\frakp(\alpha)<0\\\frakp \ne \frakp_\infty}} d^nv_\frakp(\alpha) - d^n \cdot
			\begin{cases}
			v_\infty(\alpha), &\mbox{ if } v_\infty(\alpha) < 0\\
			-1/d, &\mbox{ if } v_\infty(\alpha) \ge 0
			\end{cases}\\
		&= d^n \cdot \left(-\sum_{v_\frakp(\alpha) < 0} v_\frakp(\alpha) + 
			\begin{cases}
			0, &\mbox{ if } v_\infty(\alpha) < 0\\
			1/d, &\mbox{ if } v_\infty(\alpha) \ge 0
			\end{cases}
			\right)\\
		&= d^n \cdot
			\begin{cases}
				h(\alpha), &\mbox{ if } v_\infty(\alpha) < 0\\
				h(\alpha) + 1/d, &\mbox{ if } v_\infty(\alpha) \ge 0.
			\end{cases}
	\end{align*}
\end{proof}

The following description of the canonical height for points in $\calK$ now follows immediately from the definition.

\begin{cor}\label{cor:height}
Let $\alpha \in \calK$ and $d \ge 2$. Then
	\[
		\hhat_{f_d}(\alpha) = 
			\begin{cases}
			h(\alpha), &\mbox{ if } v_\infty(\alpha) < 0\\
			h(\alpha) + 1/d, &\mbox{ if } v_\infty(\alpha) \ge 0.
			\end{cases}
	\]
Thus, for all $n \ge 1$, we have $h(f_d^n(\alpha)) = d^n \hhat_{f_d}(\alpha)$.
\end{cor}

\subsection{Dynatomic polynomials for $f_d$}\label{sec:dyn} Throughout the article, we will require certain properties of the dynatomic polynomials for the maps $f_d(z) = z^d + t$. Suppose $x,c \in K$ are such that $x$ has period $N$ for $f_{d,c}(z) = z^d + c$. Then $(x,c)$ is a solution to the equation $f_{d,c}^N(x) - x = 0$. However, this equation is also satisfied whenever $x$ has period dividing $N$ for $f_{d,c}$. We therefore define the $N$th \textbf{dynatomic polynomial} for $f_d$ to be the polynomial
	\[
		\Phi_N(z,t) := \prod_{n \mid N} (f_d^n(z) - z)^{\mu(N/n)} \in \bbZ[z,t],
	\]
where $\mu$ is the M\"{o}bius function. (To ease notation, we omit the dependence on $d$.) The dynatomic polynomials give a natural factorization $f_d^N(z) - z = \prod_{n\mid N} \Phi_n(z,t)$. If $x$ has period $N$ for $f_{d,c}$, then $\Phi_N(x,c) = 0$, and for each $N \ge 1$ the converse is true for all but finitely many pairs $(x,c)$. That $\Phi_N(z,t)$ is indeed a polynomial is shown in \cite[Thm. 3.1]{morton/silverman:1995}; see also \cite[Thm. 4.5]{silverman:2007}. For each $N \ge 1$, we set 
	\[
		D(N) := \deg_z \Phi_N(z,t) = \sum_{n \mid N} \mu(N/n)d^n.
	\]
It is not difficult to verify that $\Phi_N(z,t)$ is monic in both $z$ and $t$, that $\deg_t \Phi_N(z,t) = D(N)/d$, and that
	\[
	\Phi_N(z,t) = z^{D(N)} + \mbox{(terms of lower total degree)}.
	\]
In particular, this implies that if $\frakp \in \calM_\calK$ is a pole of $\alpha \in \calK$, or if $\frakp = \frakp_\infty$, then
	\begin{equation}\label{eq:poles}
	v_\frakp(\Phi_N(\alpha,t)) = \min\left\{D(N)v_\frakp(\alpha), \frac{D(N)}{d}v_\frakp(t)\right\} < 0.
	\end{equation}
	
\begin{lem}\label{lem:htPhiN}
Let $\alpha \in \calK$ and $N \ge 1$. Then $v_\infty(\Phi_N(\alpha,t)) < 0$ and
	\[
		h(\Phi_N(\alpha,t)) = D(N) \cdot \hhat_{f_d}(\alpha).
	\]
In particular, $\Phi_N(\alpha,t)$ vanishes at precisely $D(N) \cdot \hhat_{f_d}(\alpha)$ \emph{finite} places, counted with multiplicity.
\end{lem}

\begin{proof}
Since $\Phi_N(z,t)$ is a polynomial in $z$ and $t$, if $\frakp$ is a pole of $\Phi_N(\alpha,t)$, then $\frakp = \frakp_\infty$ or $\frakp$ is a pole of $\alpha$.
It then follows from \eqref{eq:poles} that the poles of $\Phi_N(\alpha,t)$ are \emph{precisely} $\frakp_\infty$ and the poles of $\alpha$. Therefore
	\begin{align*}
	h(\Phi_N(\alpha,t))
		&= -\sum_{\substack{v_\frakp(\alpha) < 0\\\text{or } \frakp = \frakp_\infty}} \min\left\{D(N)v_\frakp(\alpha), \frac{D(N)}{d}v_\frakp(t)\right\}\\
		&= -\sum_{\substack{v_\frakp(\alpha) < 0\\\frakp \ne \frakp_\infty}} D(N)v_\frakp(\alpha) - 
			\begin{cases}
			D(N)v_\infty(\alpha), &\mbox{ if } v_\infty(\alpha) < 0\\
			-D(N)/d, &\mbox{ if } v_\infty(\alpha) \ge 0
			\end{cases}\\
		&= D(N) \cdot 
			\begin{cases}
			h(\alpha), &\mbox{ if } v_\infty(\alpha) < 0\\
			h(\alpha) + 1/d, &\mbox{ if } v_\infty(\alpha) \ge 0
			\end{cases}\\
		&= D(N) \cdot \hhat_{f_d}(\alpha).
	\end{align*}
\end{proof}

Finally, we record the following geometric result:
\begin{thm}\label{thm:Y1(N)}
For each integer $N \ge 1$ and $d \ge 2$, the affine plane curve $\{\Phi_N(z,t) = 0\}$ is smooth and irreducible over $K$.
\end{thm}
Theorem~\ref{thm:Y1(N)} was originally proven in the $d = 2$ case by Douady and Hubbard (smoothness; \cite[\textsection XIV]{douady/hubbard:1985}) and Bousch (irreducibility; \cite[Thm. 1 (\textsection 3)]{bousch:1992}, with a subsequent proof by Buff and Lei \cite[Thm. 3.1]{buff/lei:2014}. For $d \ge 2$, irreducibility was proven by Lau and Schleicher \cite[Thm. 4.1]{lau/schleicher:1994} using analytic methods and by Morton \cite[Cor. 2]{morton:1996} using algebraic methods, while both irreducibility and smoothness were later proven by Gao and Ou \cite[Thms. 1.1, 1.2]{gao/ou:2014} using the methods of Buff-Lei. The theorem was originally proven over $\bbC$, but the Lefschetz principle allows us to extend the result to arbitrary fields of characteristic zero.

\subsection{The $abc$-theorem for function fields}
Our main tool for proving the general case of Theorem~\ref{thm:main} is the $abc$-theorem for function fields due to Mason and Stothers \cite{mason:1984, stothers:1981}; see also \cite{silverman:1984} and \cite[Thm. F.3.6]{hindry/silverman:2000}.

\begin{thm}\label{thm:abc}
Let $\calL/\calK$ be a finite extension, and let $g_\calL$ be the genus of $\calL$. Let $u \in \calL \setminus K$, and define $S \subset \calM_\calL$ to be the set of places $\frakq$ for which $v_\frakq(u) \ne 0$ or $v_\frakq(1-u) \ne 0$. Then
	\[
		h_\calL(u) \le 2g_\calL - 2 + |S|.
	\]
\end{thm}

\section{An elementary reduction}\label{sec:reduction}

For the majority of this article, we focus on the case of Theorem~\ref{thm:main} in which $M$ is equal to 1. In this section, we justify this approach: First, we prove the theorem when $M = 0$, and then we show how the $M \ge 1$ case may essentially be reduced to $M = 1$.

\subsection{Periodic points}
In order to have $\alpha$ not realize portrait $(0,N)$ for $f_d$, it must be the case that whenever $\Phi_N(\alpha,t)$ vanishes, so too does $\Phi_n(\alpha,t)$ for some proper divisor $n$ of $N$.

\begin{lem}\label{lem:bifurcation}
Fix integers $N \ge 1$ and $d \ge 2$.
	\begin{enumerate}
	\item Let $x,c \in K$, and suppose $\Phi_N(x,c) = \Phi_n(x,c) = 0$ for some proper divisor $n$ of $N$. Then
		\[
			\left. \frac{\partial \Phi_N(z,t)}{\partial z}\right|_{(x,c)} = 0.
		\]
	\item There are strictly fewer than $D(N)$ elements $c \in K$ for which there exists $x \in K$ with $\Phi_N(x,c) = \Phi_n(x,c) = 0$ for some proper divisor $n$ of $N$.
	\end{enumerate}
\end{lem}

\begin{proof}
For part (A), see \cite[Thm. 2.4]{morton/patel:1994}; for part (B), see \cite[Cor. 3.3]{morton/vivaldi:1995}.
\end{proof}
We may now prove the $M = 0$ case of Theorem~\ref{thm:main}.
\begin{prop}\label{prop:periodic}
Let $\alpha \in \calK$, and let $N \ge 1$ and $d \ge 2$ be integers. Then $\alpha$ realizes portrait $(0,N)$ for $f_d$ if and only if $(\alpha,N,d) \ne (-1/2,2,2)$.
\end{prop}

\begin{proof}
For $\alpha \in K$, the result follows from Theorem~\ref{thm:old}, so we assume that $\alpha \in \calK \setminus K$. In this case, we have $\hhat_{f_d}(\alpha) \ge h(\alpha) \ge 1$, so it follows from Lemma~\ref{lem:htPhiN} that the number of places, counted with multiplicity, at which $\Phi_N(\alpha,t)$ vanishes is at least $D(N)$. Now suppose $\frakp_c$ is a place at which $\Phi_N(\alpha,t)$ vanishes, but $\alpha$ has period $n < N$ modulo $\frakp_c$. By Lemma~\ref{lem:bifurcation}(B), there are fewer than $D(N)$ such places, so to prove the proposition it suffices to show that $\Phi_N(\alpha,t)$ vanishes to order one at each such place. By Lemma~\ref{lem:bifurcation}(A), we have
	\[
		\left. \frac{\partial \Phi_N(z,t)}{\partial z} \right|_{(\alpha(c),c)} = 0.
	\]
Here we write $\alpha(c)$ for the reduction of $\alpha$ modulo $\frakp_c$, since this is the image of $c$ under $\alpha$ if we consider $\alpha$ as a rational map. Since the affine curve $\{\Phi_N(z,t) = 0\}$ is smooth, we must also have
	\[
		\left. \frac{\partial \Phi_N(z,t)}{\partial t} \right|_{(\alpha(c),c)} \ne 0.
	\]
This implies that
	\[
		\left. \frac{\partial \Phi_N(\alpha,t)}{\partial t} \right|_{t = c}
			= \left. \frac{\partial \Phi_N(z,t)}{\partial z}\right|_{(\alpha(c),c)} \cdot \alpha'(c) + \left. \frac{\partial \Phi_N(z,t)}{\partial t} \right|_{(\alpha(c),c)} \ne 0,
	\]
so $t = c$ is a simple root of $\Phi_N(\alpha,t)$; i.e., $\Phi_N(\alpha,t)$ vanishes to order one at $\frakp_c$, completing the proof.
\end{proof}

\subsection{Strictly preperiodic points} As mentioned previously, we will generally restrict our attention to the case $M = 1$. We now justify this approach.

\begin{lem}\label{lem:reduction}
Fix $\alpha \in \calK$ and integers $M \ge 2$, $N \ge 1$, and $d \ge 2$. Then $\alpha$ realizes portrait $(M,N)$ for $f_d$ if and only if $f_d^{M-1}(\alpha)$ realizes portrait $(1,N)$ for $f_d$.
\end{lem}

\begin{proof}
We simply note that both statements are equivalent to the statement that, at some place $\frakp \in \calM_\calK \setminus \{\frakp_\infty\}$, $f_d^M(\alpha)$ reduces to a point of period $M$ while $f_d^{M-1}(\alpha)$ does not.
\end{proof}

We also record a useful characterization of points in $K$ of portrait $(1,N)$.

\begin{lem}\label{lem:(1,N)}
Fix $x,c \in K$, $d \ge 2$, and $N \ge 1$. Then $x$ has portrait $(1,N)$ for $f_{d,c}$ if and only if $x \ne 0$ and $\zeta x$ has period $N$ for $f_{d,c}$ for some $d$th root of unity $\zeta \ne 1$.
\end{lem}

\begin{proof}
We first note that two points $x$ and $y$ have the same image under $f_{d,c}$ if and only if $y = \zeta x$ for some $d$th root of unity $\zeta$.

Suppose $x$ has portrait $(1,N)$ for $f_{d,c}$. Then $x$ is not periodic, but $f_{d,c}(x)$ has period $N$ and therefore has exactly one preimage $y$ with period $N$. We can write $y = \zeta x$ for some $d$th root of unity $\zeta$, and since $x \ne y = \zeta x$, we must have $\zeta \ne 1$ and $x \ne 0$.

Now suppose $x \ne 0$ and $\zeta x$ has period $N$ for $f_{d,c}$ for some $d$th root of unity $\zeta \ne 1$. Since $\zeta x$ has period $N$, so must $f_{d,c}(\zeta x) = f_{d,c}(x)$. However, since $x \ne 0$ and $\zeta \ne 1$, we have $x \ne \zeta x$. Thus $x$ is not periodic, and therefore $x$ has portrait $(1,N)$.
\end{proof}

\begin{cor}\label{cor:zero}
Fix integers $d \ge 2$ and $N \ge 1$. Then $0$ does not realize portrait $(1,N)$ for $f_d$.
\end{cor}

\section{The degree $d \ge 3$ case}\label{sec:deg3}

In this section, we show that if $d \ge 3$ and $\alpha \ne 0$, then $\alpha$ realizes portrait $(1,N)$ for every $N \ge 1$. We then use this to prove the $d \ge 3$ case of Theorem~\ref{thm:main}; see Proposition~\ref{prop:main3} below.

Fix $\alpha \in \calK^\times$, integers $d \ge 3$ and $N \ge 1$, and a primitive $d$th root of unity $\zeta$. Define the polynomial
	\[
		\sigma(z) := \frac{1}{\zeta - 1} \left( \frac{1}{\zeta\alpha} z - 1 \right) \in \calK[z],
	\]
which maps $\zeta\alpha$, $\zeta^2\alpha$, and $\infty$ to 0, 1, and $\infty$, respectively. Set $\gamma := f_d^N(\alpha)$, and define
	\begin{align*}
	\calA &:= \{\frakp \in \calM_\calK : v_\frakp(\sigma(\gamma)) \ne 0 \mbox{ or } v_\frakp(\sigma(\gamma) - 1) > 0\};\\
	\calB &:= \{\frakp_\infty\} \cup \{\frakp \in \calM_\calK : v_\frakp(\alpha) \ne 0 \mbox{ or } v_\frakp(\Phi_n(\zeta^k\alpha)) > 0 \mbox{ for some } n < N \mbox{ and } k \in \{1,2\}\}.
	\end{align*}

\begin{lem}\label{lem:portrait3}
If $\frakp \in \calA \setminus \calB$, then $\alpha$ has portrait $(1,N)$ for $f_d$ modulo $\frakp$.
\end{lem}

\begin{proof}
Let $\frakp \in \calA \setminus \calB$. Since $v_\frakp(\alpha) = 0$, the map $\sigma$ has good reduction --- hence remains invertible --- modulo $\frakp$. Since $\frakp \in \calA$, $\sigma(\gamma)$ reduces to 0, 1, or $\infty$ modulo $\frakp$, which implies that $\gamma$ reduces to $\zeta\alpha$, $\zeta^2\alpha$, or $\infty$ modulo $\frakp$. Since the only poles of $\gamma = f_d^N(\alpha)$ are $\frakp_\infty$ and the poles of $\alpha$, and since such places lie in $\calB$, it must be the case that $f_d^N(\alpha) \equiv \zeta^k\alpha \pmod \frakp$ for some $k \in \{1,2\}$. 

Since $\zeta^k\alpha \equiv f_d^N(\alpha) = f_d^N(\zeta^k\alpha)$ (mod $\frakp$) and $\frakp$ is not a pole of $\alpha$, $\zeta^k\alpha$ reduces to a finite point of period dividing $N$ for $f_d$ modulo $\frakp$, and the period must equal $N$ since $\Phi_n(\zeta^k\alpha) \not \equiv 0 \pmod \frakp$ for all $n < N$. Finally, since $\alpha \not \equiv 0 \pmod \frakp$, $\alpha$ has portrait $(1,N)$ modulo $\frakp$ by Lemma~\ref{lem:(1,N)}.
\end{proof}

\begin{lem}\label{lem:nonempty3}
The set $\calA \setminus \calB$ is nonempty.
\end{lem}

\begin{proof}
We first get an upper bound for $|\calA|$. In order to apply Theorem~\ref{thm:abc} with $u = \sigma(\gamma)$, we first verify that $\sigma(\gamma) \not\in K$. Indeed, suppose $\sigma(\gamma) = \lambda \in K$. Then
	\[ f_d^N(\alpha) = \gamma = \sigma^{-1}(\lambda) = \zeta((\zeta - 1)\lambda + 1)\cdot \alpha, \]
so $f_d^N(\alpha)$ is a constant multiple of $\alpha$. However, this implies that $h(f_d^N(\alpha)) \le h(\alpha)$, contradicting Lemma~\ref{lem:hhat_ineq}. Therefore $\sigma(\gamma) \not \in K$, so we may apply Theorem~\ref{thm:abc} to get
	\[
		|\calA| \ge h(\sigma(\gamma)) + 2.
	\]
Since $h(\sigma(\gamma)) = h(\gamma/\alpha) \ge h(\gamma) - h(\alpha)$ and $h(\gamma) = h(f_d^N(\alpha)) = d^N \hhat_{f_d}(\alpha)$, we have
	\[
		|\calA| \ge h(\gamma) - h(\alpha) + 2 = d^N \hhat_{f_d}(\alpha) - h(\alpha) + 2 \ge (d^N - 1)\hhat_{f_d}(\alpha) + 2.
	\]
On the other hand, it is straightforward to verify that
	\begin{align*}
		|\calB| \le 1 + 2h(\alpha) + \sum_{k=1}^2 \sum_{\substack{n \mid N\\n < N}} h(\Phi_n(\zeta^k\alpha))
			&= 1 + 2h(\alpha) + \sum_{k=1}^2 \sum_{\substack{n \mid N\\n < N}} \hhat_{f_d}(\zeta^k\alpha)D(n)\\
			&= 1 + 2h(\alpha) + 2\hhat_{f_d}(\alpha)\sum_{\substack{n \mid N\\n < N}} D(n)\\
			&\le 1 + \hhat_{f_d}(\alpha)\left(2 + 2\sum_{\substack{n \mid N\\n < N}} D(n)\right).
	\end{align*}
Combining these bounds on $|\calA|$ and $|\calB|$, we find that $|\calA \setminus \calB| \ge \kappa \hhat_{f_d}(\alpha) + 1$, where
	\[
	\kappa := d^N - 3 - 2\sum_{\substack{n \mid N\\n < N}} D(n).
	\]
To show that $\calA \setminus \calB$ is nonempty, it suffices to show that $\kappa \ge 0$, since $\hhat_{f_d}(\alpha) > 0$ for all $\alpha \in \calK$. Now observe that
	\[
		\sum_{\substack{n \mid N\\n < N}} D(n) \le \sum_{\substack{n \mid N\\n < N}} d^n \le \sum_{n=1}^{\lfloor N/2 \rfloor} d^n \le \frac{d}{d-1}(d^{N/2} - 1),
	\]
and therefore
	\begin{align*}
		\kappa
			\ge d^N - 3 - \frac{2d}{d-1}(d^{N/2} - 1)
			&\ge d^N - 3 - d(d^{N/2} - 1)\\
			&= d^{N/2 + 1}\left(d^{N/2 - 1} - 1\right) + d - 3.
	\end{align*}
This expression is nonnegative for all $d \ge 3$ and $N \ge 2$; on the other hand, when $N = 1$, we have $\kappa = d - 3 \ge 0$. In either case, we conclude that $\calA \setminus \calB$ is nonempty.
\end{proof}

We may now prove Theorem~\ref{thm:main} for the general case $d \ge 3$.

\begin{prop}\label{prop:main3}
Let $(\alpha,M,N,d) \in \calK \times \bbZ^3$ with $M \ge 0$, $N \ge 1$, and $d \ge 3$. Then $\alpha$ realizes portrait $(M,N)$ for $f_d$ if and only if $(\alpha,M) \ne (0,1)$.
\end{prop}

\begin{proof}
For $M = 0$, this follows from Proposition~\ref{prop:periodic}. Corollary~\ref{cor:zero} says that 0 does not realize portrait $(1,N)$ for any $N \ge 1$ and $d \ge 2$, so suppose $\alpha \ne 0$. Letting $\calA$ and $\calB$ be as above, the set $\calA \setminus \calB$ is nonempty, thus $\alpha$ realizes portrait $(1,N)$ for $f_d$, proving the statement for $M = 1$.

Finally, let $M \ge 2$. Since $0 \notin f_d(\calK)$, the point $f_d^{M-1}(\alpha)$ is nonzero, so $f_d^{M-1}(\alpha)$ realizes portrait $(1,N)$ for $f_d$. By Lemma~\ref{lem:reduction}, we conclude that $\alpha$ realizes portrait $(M,N)$ for $f_d$.
\end{proof}

\section{The degree $d = 2$ case}\label{sec:deg2}

We henceforth drop the subscript and write $f = f_2$ and $f_c = f_{2,c}$. To prove Theorem~\ref{thm:main} when $N \ge 4$, we proceed much like in \textsection \ref{sec:deg3}; however, we require different methods when $N \le 3$, so we consider these cases separately.

\subsection{$N \ge 4$}
The proof of Proposition~\ref{prop:main3} in the previous section relied on fixing two preimages of $f_d(\alpha)$ \emph{different from $\alpha$ itself}, then counting the number of places at which $f_d^N(\alpha)$ had the same reduction as one of those two preimages. When $d = 2$, however, there is only \emph{one} preimage of $f(\alpha)$ different from $\alpha$ (namely, $-\alpha$), so we require a minor modification of the technique from \textsection \ref{sec:deg3}.

Fix $\alpha \in \calK$ and $N \ge 1$. Let $\eta \in \bar{\calK}$ satisfy $\eta^2 + t = -\alpha$, set $\calL := \calK(\eta)$, and set $\delta := [\calL:\calK] \in \{1,2\}$. Define the polynomial
	\[
		\sigma(z) := -\frac{1}{2}\left(\frac{z}{\eta} - 1\right) \in \calK[z],
	\]
which maps $\eta$, $-\eta$, and $\infty$ to 0, 1, and $\infty$, respectively. Set $\gamma := f^{N-1}(\alpha)$, and define
	\begin{align*}
	\calA &:= \{\frakq \in \calM_\calL : v_\frakq(\sigma(\gamma)) \ne 0 \mbox{ or } v_\frakq(\sigma(\gamma) - 1) > 0\};\\
	\calB &:= \calM_{\calL,\infty} \cup \{\frakq \in \calM_\calL : v_\frakq(\eta) \ne 0, \ v_\frakq(\alpha) > 0, \mbox{ or } v_\frakq(\Phi_n(-\alpha)) > 0 \mbox{ for some } n < N\}.
	\end{align*}
By a proof similar to that of Lemma~\ref{lem:portrait3}, if $\frakq \in \calA \setminus \calB$, then $\alpha$ has portrait $(1,N)$ for $f$ modulo $\frakq$. We now show that there exists at least one such place.

\begin{lem}\label{lem:nonempty2}
The set $\calA \setminus \calB$ is nonempty.
\end{lem}

\begin{proof}
We first get a lower bound for $|\calA|$. By an argument similar to the beginning of the proof of Lemma~\ref{lem:nonempty3}, we have $\sigma(\gamma) \not \in K$, so we apply Theorem~\ref{thm:abc} to get
	\[
		|\calA| \ge h_\calL(\sigma(\gamma)) - (2g_\calL - 2).
	\]		
We get a lower bound on $h_\calL(\sigma(\gamma))$ by noting that $h_\calL(\sigma(\gamma)) = h_\calL(\gamma/\eta) \ge h_\calL(\gamma) - h_\calL(\eta)$; since also $h_\calL(\gamma) = \delta \cdot h(\gamma) = \delta \cdot 2^{N-1}\hhat_f(\alpha)$, we have $h_\calL(\sigma(\gamma)) \ge \delta(2^{N-1}\hhat_f(\alpha) - h(\eta))$.

We also require an upper bound on $2g_\calL - 2$. Let $\calR_{\calL/\calK}$ be the ramification divisor of the extension $\calL/\calK$. Since $\calL/\calK$ is generated by $\eta = \sqrt{-(\alpha + t)}$, the only places in $\calM_\calK$ over which $\calL$ may ramify are zeroes and poles of $\alpha + t$. Thus $\deg \calR_{\calL/\calK} \le (\delta - 1)\cdot 2h(\alpha + t) = 4(\delta - 1)h(\eta)$, so it follows from Riemann-Hurwitz that $2g_\calL - 2 \le -2\delta + 4(\delta - 1)h(\eta)$. Therefore
	\begin{align*}
		|\calA| &\ge \delta(2^{N-1}\hhat_f(\alpha) - h(\eta)) - (-2\delta + 4(\delta - 1)h(\eta)) = \delta \cdot \left( 2^{N-1}\hhat_f(\alpha) -\left(5 - \frac{4}{\delta}\right)h(\eta) + 2 \right).
	\end{align*}
On the other hand, it is straightforward to verify that
	\begin{align*}
		|\calB| \le \#\calM_{\calL,\infty} + 2h_\calL(\eta) + h_\calL(\alpha) + \sum_{\substack{n \mid N\\n < N}} h_\calL(\Phi_n(-\alpha)) &\le \delta + 2\delta h(\eta) + \delta h(\alpha) + \sum_{\substack{n \mid N\\n < N}} \delta h(\Phi_n(-\alpha))\\
			&= \delta \left(1 + 2h(\eta) + h(\alpha) + \hhat_f(\alpha)\sum_{\substack{n \mid N\\n < N}} D(n)\right).
	\end{align*}
Combining these bounds on $|\calA|$ and $|\calB|$ yields
	\[
	|\calA \setminus \calB| \ge \delta\cdot\left[ \left(2^{N-1} - \sum_{\substack{n \mid N\\n < N}} D(n)\right)\hhat_f(\alpha) - h(\alpha) - \left(7 - \frac{4}{\delta}\right)h(\eta) + 1 \right].
	\]
It remains to show that this bound is positive for all $N \ge 4$ and $\alpha \in \calK^\times$.

First, suppose $v_\infty(\alpha) < 0$. Then $\hhat_f(\alpha) = h(\alpha)$ and $h(\eta) = \frac{1}{2}h(\alpha + t) \ge \frac{1}{2}(h(\alpha) - 1)$. Therefore
	\begin{align*}
	|\calA \setminus \calB|
		&\ge \delta\cdot\left[\left(2^{N-1} - \sum_{\substack{n \mid N\\n < N}} D(n)\right)h(\alpha) - h(\alpha) - \left(7 - \frac{4}{\delta}\right)\frac{1}{2}(h(\alpha) - 1) + 1\right] \\
		&= \delta\cdot\left[\left(2^{N-1} - \frac{9}{2} + \frac{2}{\delta} - \sum_{\substack{n \mid N\\n < N}} D(n)\right) h(\alpha) + \left(\frac{9}{2} - \frac{2}{\delta}\right) \right]\\
		&\ge \delta\cdot\left[\left(2^{N-1} - \frac{7}{2} - \sum_{\substack{n \mid N\\n < N}} D(n)\right) h(\alpha) + \frac{5}{2}\right],
	\end{align*}
where the last inequality follows from the fact that $\delta \in \{1,2\}$. Since $h(\alpha) \ge 0$ for all $\alpha \in \calK$, it suffices to show that the quantity
	\begin{equation}\label{eq:coeff}
	2^{N-1} - \frac{7}{2} - \sum_{\substack{n \mid N\\n < N}} D(n)
	\end{equation}
is non-negative for all $N \ge 4$. We bound the sum just as in the proof of Lemma~\ref{lem:nonempty3} to get
	\[
	2^{N-1} - \frac{7}{2} - \sum_{\substack{n \mid N\\n < N}} D(n)
		\ge 2^{N-1} - \frac{7}{2} - 2(2^{N/2} - 1) = 2^{N/2 + 1}(2^{N/2 - 2} - 1) - \frac{3}{2}.
	\]
The rightmost expression is positive for $N \ge 5$ but negative for $N = 4$; however, for $N = 4$ one can show directly that \eqref{eq:coeff} is positive since $D(1) = D(2) = 2$. (Note that \eqref{eq:coeff} is negative for $N \le 3$.) Thus $|\calA \setminus \calB|$ is positive for all $N \ge 4$.

Now suppose that $v_\infty(\alpha) \ge 0$, in which case we have $\hhat_f(\alpha) = h(\alpha) + \frac{1}{2}$ and $h(\eta) = \frac{1}{2}h(\alpha + t) = \frac{1}{2}(h(\alpha) + 1)$. By an estimate similar to the previous case, we find that
	\[
	|\calA \setminus \calB| \ge \delta\cdot\left[\left(2^{N-1} - \frac{7}{2} - \sum_{\substack{n \mid N\\n < N}} D(n)\right) h(\alpha) + \frac{1}{2}\left(2^{N-1} - \sum_{\substack{n \mid N\\n < N}} D(n) - 3 \right)\right].
	\]
Since we have already shown that \eqref{eq:coeff} is positive for all $N \ge 4$, it follows that this expression is positive for all $N \ge 4$ as well.

In both cases, our lower bound on $|\calA \setminus \calB|$ is positive, so $\calA \setminus \calB$ is nonempty.
\end{proof}

The same argument as for Proposition~\ref{prop:main3} yields the $d = 2$, $N \ge 4$ case of Theorem~\ref{thm:main}.
\begin{prop}\label{prop:maind2N4}
Let $(\alpha,M,N) \in \calK \times \bbZ^2$ with $M \ge 0$ and $N \ge 4$. Then $\alpha$ realizes portrait $(M,N)$ for $f_2$ if and only if $(\alpha,M) \ne (0,1)$.
\end{prop}

\subsection{$N = 3$} Since the technique used for periods $N \ge 4$ gives a negative lower bound when $N \le 3$, we again require a different method. For $N = 3$, we consider the third dynatomic polynomial
	\begin{align*}
	\Phi_3(z,t) &= \frac{f^3(z) - z}{f(z) - z}\\
		&= z^6 + z^5 + (3t + 1)z^4 + (2t + 1)z^3 + (3t^2 + 3t + 1)z^2 +
	    (t + 1)^2z + t^3 + 2t^2 + t + 1.
	\end{align*}
The roots of $\Phi_3(z,t)$ in $\calKbar$ are the points of period 3 for $f$, which fall naturally into two 3-cycles. Let $\eta$ be one such root, and let $\calL := \calK(\eta)$; since $\Phi_3(z,t)$ is irreducible, we have $[\calL:\calK] = 6$. Set $\eta_1 := \eta$, $\eta_2 := f(\eta)$, and $\eta_3 := f^2(\eta)$, and note that each $\eta_i$ is a root of $\Phi_3(z,t)$ that generates $\calL/\calK$. Denote by $\calX_\calL$ the normalization of the projective closure of the affine curve $\{\Phi_3(z,t) = 0\}$, and note that the extension $\calL/\calK$ corresponds to the morphism $\calX_\calL \to \bbP^1$ mapping $(z,t) \mapsto t$. In his thesis, Bousch gave a general formula \cite[\textsection 3, Thm. 2]{bousch:1992} for the genera of the curves $\{\Phi_N(z,t) = 0\}$, and in this case we have $g_\calL = 0$. (For an explicit parametrization of the curve $\calX_\calL$, see \cite{walde/russo:1994}.)

By a result of Morton \cite[Prop. 10]{morton:1996} for more general dynatomic curves, $\calM_{\calL,\infty}$ consists of three places, each of which has ramification degree two over $\frakp_\infty$. To describe the ramification at finite places, Morton shows in \cite[p. 358]{morton:1992} that the discriminant of $\Phi_3(z,t) \in \calK[z]$ is given by
	\begin{equation}\label{eq:disc}
		\disc \Phi_3(z,t) = \Delta_{3,1}^2\Delta_{3,3}^3,
	\end{equation}
where $\Delta_{3,1} := \Res_z(\Phi_3(z,t),\Phi_1(z,t)) = -(16t^2 + 4t + 7)$ and $\Delta_{3,3} := -(4t + 7)$. The roots $c$ of $\Delta_{3,1}$ correspond to maps $f_c$ for which one cycle of length three collapses to a fixed point, while the roots of $\Delta_{3,3}$ correspond to maps where the two 3-cycles collide to form a single 3-cycle.

Of particular relevance for us is the polynomial $\Delta_{3,1}$. Let $c_1,c_2 \in K$ be the two roots of $\Delta_{3,1}$. For each $i$, $\calM_{\calL,c_i}$ consists of four places: $\bar{\frakq_i}$, which has ramification degree three, and $\frakq_{i,1}$, $\frakq_{i,2}$, and $\frakq_{i,3}$, each of which is unramified --- see the proof of \cite[Prop. 9]{morton:1996}. The places $\bar{\frakq_1}$ and $\bar{\frakq_2}$ are precisely the places at which $\eta$ has period one; that is, the finite places at which $\eta_1$, $\eta_2$, and $\eta_3$ have the same reduction.

We briefly explain this geometrically: Let $c \in K$ be a root of $\Delta_{3,1}$. Instead of having two 3-cycles, $f_c$ has one 3-cycle $\{x_1,x_2,x_3\}$ and a fixed point $\bar{x}$ which satisfies $\Phi_3(\bar{x},c) = 0$. Thus the only points on $\calX_\calL$ that map to $c \in \bbP^1$ are $(\bar{x},c)$ and $(x_j,c)$ for $j \in \{1,2,3\}$. Since $(\bar{x},c)$ is fixed by the order three automorphism $(z,t) \mapsto (f(z),t)$, this point ramifies over $c$, while the three points $(x_j,c)$ are unramified.

\begin{lem}\label{lem:val2}
Let $c$ be a root of $\Delta_{3,1}$, and let $\bar{\frakq} \in \calM_\calL$ be the unique place ramified over $\frakp_c$. For each $1 \le i < j \le 3$,
	\begin{enumerate}
	\item $v_{\bar{\frakq}}(\eta_i) = 0$;
	\item $v_{\bar{\frakq}}(\eta_i - \eta_j) = 1$;
	\item $v_\frakq(\eta_i) = -1$ for each place $\frakq \in \calM_{\calL,\infty}$;
	\item There exists a place $\frakq_{i,j} \in \calM_{\calL,\infty}$ such that, for $\frakq \in \calM_{\calL,\infty}$,
		\[
			v_\frakq(\eta_i - \eta_j) =
				\begin{cases}
				0, &\mbox{ if } \frakq = \frakq_{i,j};\\
				-1, &\mbox{ otherwise}.
				\end{cases}
		\]
	Moreover, $\frakq_{1,2}$, $\frakq_{1,3}$, and $\frakq_{2,3}$ are distinct.
	\end{enumerate}
%The natural generalization for $f_d$ and $\Delta_{N,n}$ are also true. The proof is slightly more complicated.
\end{lem}

\begin{proof}
Since $\Phi_3(z,t)$ is monic in $z$, the only poles of $\eta_1$, $\eta_2$, and $\eta_3$ must lie above $\frakp_\infty$. Therefore, to prove part (A), it suffices to show that none of the $\eta_i$ vanish at $\bar{\frakq}$. This follows by noting that $\Phi_3(0,c) \ne 0$.

For each $i,j$, the points $\eta_i$ and $\eta_j$ have the same reduction modulo $\bar{\frakq}$, so $v_{\bar{\frakq}}(\eta_i - \eta_j) \ge 1$. Now, we observe that the product
	\[
		\prod_{1 \le i < j \le 3} (\eta_i - \eta_j)^2
	\]
divides $\disc \Phi_3(z,t)$, which then implies that
	\[
		\sum_{1 \le i < j \le 3} 2v_{\bar\frakq}(\eta_i - \eta_j) \le v_{\bar{\frakq}}(\disc \Phi_3(z,t)) = v_{\bar{\frakq}}(\Delta_{3,1}^2\Delta_{3,3}^3).
	\]
This sum has three terms, and each term is at least 2, hence the sum is at least 6. Also, the polynomial $\Delta_{3,3}$ does not vanish at $\bar\frakq$; moreover, if we factor $\Delta_{3,1}$ into linear factors, only the factor $(t - c)$ vanishes at $\bar\frakq$. This implies that
	\[
	6 \le \sum_{1 \le i < j \le 3} 2v_{\bar{\frakq}}(\eta_i - \eta_j) \le 2v_{\bar\frakq}(t - c) = 2e_{\bar\frakq/\frakp_c} = 6.
	\]
We therefore have equality throughout, so $v_{\bar\frakq}(\eta_i - \eta_j) = 1$ for each $1 \le i < j \le 3$, proving (B).

Now fix $\frakq \in \calM_{\calL,\infty}$ and $1 \le i \le 3$. Note that $v_\frakq(t) = -e_{\frakq/\frakp_\infty} = -2$. If $v_\frakq(\eta_i) < -1$, then an induction argument shows that $v_\frakq(f^n(\eta_i)) = 2^nv_\frakq(\eta_i)$ for each $n \in \bbN$; in particular, we have $v_\frakq(\eta_i) = v_\frakq(f^3(\eta_i)) = 8v_\frakq(\eta_i)$, a contradiction. If instead $v_\frakq(\eta_i) > -1$, then $v_\frakq(f(\eta_i)) = v_\frakq(t) < -1$, so we reduce to the previous case to get a contradiction. Therefore $v_\frakq(\eta_i) = 1$, proving (C).

The only finite zeroes of $\eta_i - \eta_j$ are the two simple zeroes at $\frakq_1$ and $\frakq_2$, so there must be at least two poles of $\eta_i - \eta_j$. By part (C), these must be simple poles lying above $\frakp_\infty$. Comparing the degrees of its zero and pole divisors, $\eta_i - \eta_j$ must have a simple pole at \emph{precisely} two places above $\frakp_\infty$; moreover, if we let $\frakq_{i,j}$ be the remaining infinite place, then $v_{\frakq_{i,j}}(\eta_i - \eta_j) = 0$.

Finally, suppose $(i,j) \ne (i',j')$ but $\frakq_{i,j} = \frakq_{i',j'}$. Reordering the indices if necessary, we may assume that $\frakq_{1,2} = \frakq_{1,3}$. Since $\eta_1 - \eta_2$ and $\eta_1 - \eta_3$ have exactly the same zeroes and poles, with exactly the same orders, it must be that $\lambda := (\eta_1 - \eta_2)/(\eta_1 - \eta_3)$ is constant. This implies that
	\[
		0 = (\eta_1 - \eta_2) - \lambda(\eta_1 - \eta_3) = \Big(\eta - (\eta^2 + t)\Big) - \lambda\Big(\eta - \big((\eta^2 + t)^2 + t\big)\Big)
	\]
for some $\lambda \in K$, contradicting the fact that $\eta$ has degree 6 over $\calK$. We conclude that the places $\frakq_{1,2}$, $\frakq_{1,3}$, and $\frakq_{2,3}$ are distinct, completing the proof.
\end{proof}
\begin{rem}\label{rem:val2}
It follows from Lemma~\ref{lem:val2} and its proof that $h_\calL(\eta_i) = 3$ for each $1 \le i \le 3$, since the only poles of $\eta_i$ are simple poles at the three places lying above $\frakp_\infty$. Similarly, we have $h_\calL(\eta_i - \eta_j) = 2$ for each $1 \le i < j \le 3$.
\end{rem}
Now consider the affine rational map
	\[
		\sigma(z) := \frac{\eta_2 - \eta_3}{\eta_2 - \eta_1} \cdot \frac{z - \eta_1}{z - \eta_3},
	\]
which is constructed to map $\eta_1$, $\eta_2$, and $\eta_3$ to 0, 1, and $\infty$ respectively. Note that
	\[
		\sigma^{-1}(z) = \frac{\eta_3(\eta_1 - \eta_2)z + \eta_1(\eta_2 - \eta_3)}{(\eta_1 - \eta_3)z + (\eta_2 - \eta_3)}.
	\]
\begin{lem}\label{lem:ht_bound}
For all $x \in \calL$, $h_\calL(\sigma(x)) \ge h_\calL(x) - 4$.
\end{lem}
%Having $-3$ would be best possible, as seen by taking $x = \eta_1$, but this doesn't help us in the end.

\begin{proof}
By considering $y = \sigma(x)$, it suffices to show that $h_\calL(\sigma^{-1}(y)) \le h_\calL(y) + 4$ for all $y \in \calL$. A standard height argument (see \cite[pp. 90--92]{silverman:2007}) shows that for any $y \in \calL$ we have
	\[
	h_\calL(\sigma^{-1}(y))
		\le h_\calL(y) - \sum_{\frakq \in \calM_\calL} \min \{v_\frakq(\eta_3(\eta_1 - \eta_2)), v_\frakq(\eta_1(\eta_2 - \eta_3)), v_\frakq(\eta_1 - \eta_3), v_\frakq(\eta_2 - \eta_3) \},
	\]
so it suffices to show that the quantity
	\begin{equation}\label{eq:ht_est}
	\sum_{\frakq \in \calM_\calL} \min \{v_\frakq(\eta_3) + v_\frakq(\eta_1 - \eta_2), v_\frakq(\eta_1) + v_\frakq(\eta_2 - \eta_3), v_\frakq(\eta_1 - \eta_3), v_\frakq(\eta_2 - \eta_3) \}
	\end{equation}
is equal to $-4$. We now determine the places $\frakq$ at which the `min' term is nonzero.

If the minimum is positive at $\frakq$, then $\frakq \in \{\frakq_1,\frakq_2\}$. By Lemma~\ref{lem:val2}, at each such $\frakq$ and for each $1 \le i < j \le 3$, we have $v_\frakq(\eta_i - \eta_j) = 1 \mbox{ and } v_\frakq(\eta_i) = 0$. Thus the contribution to \eqref{eq:ht_est} at each such place is equal to $1$, so the combined contribution at both places is equal to $2$.

If the minimum is negative at $\frakq$, then $\frakq$ necessarily lies above $\frakp_\infty$. Again applying Lemma~\ref{lem:val2}, $v_\frakq(\eta_i) = -1$ for each $1 \le i \le 3$, and exactly one of $v_\frakq(\eta_1 - \eta_2)$, $v_\frakq(\eta_1 - \eta_3)$, and $v_\frakq(\eta_2 - \eta_3)$ is nonnegative. In particular, this means that $\min\{v_\frakq(\eta_1 - \eta_2), v_\frakq(\eta_2 - \eta_3)\} = -1$, so the combined contribution to \eqref{eq:ht_est} at the three infinite places is $3 \cdot (-2) = -6$. Since we have accounted for all nonzero terms in the sum, we conclude that the expression \eqref{eq:ht_est} is equal to $-4$, as claimed.
\end{proof}

Now let $\alpha \in \calK^\times$ be arbitrary. We will show that $\alpha$ realizes portrait $(1,3)$ for $f$. Define
	\begin{align*}
	\calA &:= \{\frakq \in \calM_\calL : v_\frakq(\sigma(-\alpha)) \ne 0 \mbox{ or } v_\frakq(\sigma(-\alpha) - 1) > 0\};\\
	\calB &:= \calM_{\calL,\infty} \cup \{\bar{\frakq_1}, \bar{\frakq_2}\} \cup \{\frakq \in \calM_\calL : v_\frakq(\eta_i) > 0 \mbox{ for some } i = 1,2,3\}.
	\end{align*}

\begin{lem}
If $\frakq \in \calA \setminus \calB$, then $\alpha$ has portrait $(1,3)$ for $f$ modulo $\frakq$.
\end{lem}

\begin{proof}
Let $\frakq \in \calA \setminus \calB$. Since $\frakq \not \in \calM_{\calL,\infty}$, $f$ has good reduction at $\frakq$. Since $\eta$ has period 3 for $f$, it follows that $\eta$  has period 1 or 3 for $f$ modulo $\frakq$. The only poles of $\eta$ are places at infinity, so $\eta \not \equiv \infty \pmod \frakq$; moreover, since $\frakq \not \in \{\bar{\frakq_1},\bar{\frakq_2}\}$ we have
	\[
		f(\eta) - \eta = \eta_2 - \eta_1 \not \equiv 0 \pmod \frakq,
	\]
so $\eta$ cannot have period 1. Thus $\eta$ has period 3 for $f$ modulo $\frakq$.

The zeroes and poles of the coefficients of $\sigma$ all lie in $\calB$, so each coefficient is a unit in the residue field $k_\frakq$. Since also $\eta_1 \not \equiv \eta_3$ (mod $\frakq$), the reduction of $\sigma$ has degree one over $k_\frakq$; i.e., $\sigma$ remains invertible modulo $\frakq$. Thus, since $\sigma(-\alpha)$ reduces to 0, 1, or $\infty$ modulo $\frakq$, $-\alpha$ must reduce to $\eta_i$ for some $i \in \{1,2,3\}$. We have already shown that each $\eta_i$ has period 3 modulo $\frakq$, so the same is true for $-\alpha$. Finally, since $-\alpha \equiv \eta_i \not \equiv 0 \pmod \frakq$, $\alpha$ must reduce to a point of portrait $(1,3)$.
\end{proof}

\begin{lem}\label{lem:d2N3nonempty}
Suppose $h(\alpha) \ge 3$. Then $\calA \setminus \calB$ is nonempty.
\end{lem}

\begin{proof}
We have by Lemma~\ref{lem:ht_bound} that $h_\calL(\sigma(-\alpha)) \ge 14$, so $\sigma(-\alpha) \not \in K$. Applying Theorem~\ref{thm:abc}, it follows that the set $\calA$ has size at least $h_\calL(\sigma(-\alpha)) - (2g_\calL - 2) = h_\calL(\sigma(-\alpha)) + 2$. Again applying the bound in Lemma~\ref{lem:ht_bound} gives
	\[
		|\calA| \ge (h_\calL(\alpha) - 4) + 2 = 6h(\alpha) - 2.
	\]
We also have $|\calB| \le |\calM_{\calL,\infty}| + 2 + \sum_{i=1}^3 h_\calL(\eta_i) = 14$, which implies that $|\calA \setminus \calB| \ge 6h(\alpha) -16$. Therefore $\calA \setminus \calB$ is nonempty when $h(\alpha) \ge 3$.
\end{proof}

We have just shown that if $h(\alpha) \ge 3$, then there exists a place $\frakq \in \calM_\calL$ for which $\alpha$ has portrait $(1,3)$ modulo $\frakq$; choosing $\frakp \in \calM_\calK$ below $\frakq$, the same holds modulo $\frakp$. It remains to prove this in the case $h(\alpha) \le 2$. The constant point case $h(\alpha) = 0$ is covered by Theorem~\ref{thm:old}, so we henceforth assume $h(\alpha) \in \{1,2\}$. To handle these remaining cases --- as well as the $N = 2$ and $N = 1$ cases in Section~\ref{sec:d=2,N=2} --- we use the following consequence of Lemma~\ref{lem:(1,N)}:

\begin{cor}\label{cor:fail}
Let $\alpha \in \calK$ and $N \in \bbN$. The following are equivalent:
	\begin{enumerate}
	\item The point $\alpha$ does not realize portrait $(1,N)$ for $f$.
	\item For every place $\frakp \in \calM_\calK$ at which $\Phi_N(-\alpha,t)$ vanishes, either $\Phi_n(-\alpha,t)$ also vanishes at $\frakp$ for some proper divisor $n$ of $N$, or $\alpha$ also vanishes at $\frakp$.
	\end{enumerate}
\end{cor}

\begin{rem}\label{rem:simple}
If $\Phi_N(\beta,t)$ and $\Phi_n(\beta,t)$ both vanish at $\frakp$ for some proper divisor $n$ of $N$, then it follows from the proof of Proposition~\ref{prop:periodic} that $v_\frakp(\Phi_N(\beta,t)) = 1$.
\end{rem}

If $\alpha$ does not realize portrait $(1,3)$ for $f$, then wherever $\Phi_3(-\alpha)$ vanishes, either $\alpha$ or $\Phi_1(-\alpha)$ must vanish as well. We will show that such behavior is impossible when $h(\alpha) \in \{1,2\}$, having already handled all other cases.

\begin{lem}\label{lem:d2N3}
Let $\beta \in \calK$, and suppose $h(\beta) = 1$ or $h(\beta) = 2$. There exists a place $\frakp \in \calM_\calK \setminus \{\frakp_\infty\}$ such that $\Phi_3(\beta,t)$ vanishes at $\frakp$ but $\Phi_1(\beta,t)$ and $\beta$ do not.
\end{lem}

\begin{proof}
Suppose to the contrary that there does not exist such a place. Then, if we fix a place $\frakp = \frakp_c$ for which $v_c(\Phi_3(\beta,t)) > 0$, we must also have $v_c(\Phi_1(\beta,t)) > 0$ or $v_c(\beta) > 0$.

If $v_c(\Phi_1(\beta,t)) > 0$, then it must be that $c \in \{c_1,c_2\}$, where $c_1$ and $c_2$ are as above Lemma~\ref{lem:val2}. Moreover, in this case the order of vanishing of $\Phi_3(\beta,t)$ at $\frakp_c$ must equal one by Remark~\ref{rem:simple}.

By Lemma~\ref{lem:htPhiN}, $\Phi_3(\beta,t)$ vanishes at precisely $6\hhat_f(\beta) \ge 6$ places, counted with multiplicity. Since $\Phi_3(\beta,t)$ can have at most simple roots at $c_1$ and $c_2$, there must be a place $\frakp_c$ at which both $\Phi_3(\beta,t)$ and $\beta$ vanish. In this case, we must have
	\[
		\Phi_3(0,c) = c^3 + 2c^2 + c + 1 = 0.
	\]
Let $C_1$, $C_2$, and $C_3$ be the three roots of $\Phi_3(0,t)$. Since $h(\beta) \le 2$, $\beta$ can have at most two roots; reordering the roots if necessary, we assume that $\Phi_3(\beta,t)$ and $\beta$ both vanish at $C_1$ and possibly $C_2$.

We have put certain restrictions on the places at which $\Phi_3(\beta,t)$ may vanish as well as the order of vanishing at each such place. Let $\rho(\beta) \in K[t]$ denote the numerator of $\Phi_3(\beta,t)$; scaling if necessary, we assume that $\rho(\beta)$ is monic. Set $R := \deg \rho(\beta) = 6\hhat_f(\beta)$.

If $\Phi_3(\beta,t)$ vanishes at $\frakp_{C_1}$ but not $\frakp_{C_2}$, then $\beta = (t - C_1)p/q$ for some $p,q \in K[t]$ with $\deg p \le 1$ and $\deg q \le 2$, and
	\begin{equation}\label{eq:nu1}
	\rho(\beta) = (t - c_1)^{\epsilon_1}(t - c_2)^{\epsilon_2}(t - C_1)^{R - \epsilon_1 - \epsilon_2}
	\end{equation}
for some $\epsilon_1,\epsilon_2 \in \{0,1\}$. If $\Phi_3(\beta,t)$ vanishes at both $\frakp_{C_1}$ and $\frakp_{C_2}$, then $\beta = (t - C_1)(t - C_2)/q$ for some $q \in K[t]$ with $\deg q \le 2$, and
	\begin{equation}\label{eq:nu2}
	\rho(\beta) = (t - c_1)^{\epsilon_1}(t - c_2)^{\epsilon_2}(t - C_1)^k(t - C_2)^{R - \epsilon_1 - \epsilon_2 - k}
	\end{equation}
for some $\epsilon_1,\epsilon_2 \in \{0,1\}$ and $1 \le k \le R - \epsilon_1 - \epsilon_2 - 1$.

The idea is to write $p$ and $q$ as polynomials with indeterminate coefficients, then compare the coefficients of both sides of \eqref{eq:nu1} (resp., \eqref{eq:nu2}). This will determine an affine scheme over $K$; if this scheme is empty, or if the only points on the scheme yield a constant map $\beta$, then we will have completed the proof of the lemma.

We illustrate this computation in one case. Suppose that $\Phi_3(\beta,t)$ vanishes at $\frakp_{C_1}$ but not $\frakp_{C_2}$, and let us suppose that $\deg q = 2$, in which case $R = 15$. Write $\beta = (t - C_1)(p_1t + p_0)/(t^2 + q_1t + q_0)$. Then $\rho(\beta)$ is a polynomial in $t$ with coefficients in $K[p_0,p_1,q_0,q_1]$, and comparing the coefficients of $\rho(\beta)$ with the coefficients of $(t - c_1)^{\epsilon_1}(t - c_2)^{\epsilon_2}(t - C_1)^{15 - \epsilon_1 - \epsilon_2}$ for each pair $(\epsilon_1,\epsilon_2) \in \{0,1\}^2$ yields four different subschemes of $\bbA^4_K = \Spec K[p_0,p_1,q_0,q_1]$. A computation in Magma \cite{magma} verifies that each of these schemes is empty.  The proof of the lemma is then completed by a number of similar computations; for the interested reader, the Magma code and output have been included as an ancillary file with this article's arXiv submission.
\end{proof}

Applying Lemma~\ref{lem:d2N3} with $\beta = -\alpha$ shows that if $h(\alpha) \in \{1,2\}$, then $\alpha$ realizes portrait $(1,3)$ for $f$; as mentioned above, we may now conclude that this holds for all $\alpha \in \calK^\times$. Finally, arguing as in the proof of Proposition~\ref{prop:main3}, we have the $d = 2$, $N = 3$ case of Theorem~\ref{thm:main}:

\begin{prop}\label{prop:maind2N3}
Let $\alpha \in \calK$, and let $M \ge 0$. Then $\alpha$ realizes portrait $(M,3)$ for $f_2$ if and only if $(\alpha,M) \ne (0,1)$.
\end{prop}

\subsection{$N \le 2$}\label{sec:d=2,N=2}
In this section, we prove Theorem~\ref{thm:main} in the $d = 2$, $N = 2$ case. The proof for $N = 1$ uses essentially the same technique, so we omit the proof in that case.

Consider the rational map
\[
\phi_2(z) := \frac{(t+1)(z - 1)}{z + t},
\]
with inverse
\[
\phi_2^{-1}(z) = -\frac{tz + (t + 1)}{z - (t + 1)}.
\]
It follows from Theorem~\ref{thm:old} that $0$, $1/2$, and $\infty$ (which is a fixed point for $f$) are the only points in $\bbP^1(K)$ that fail to realize portrait $(1,2)$ for $f$. We will show that the points in $\calK$ that fail to realize portrait $(1,2)$ are precisely the points in the orbits of $0$, $1/2$, and $\infty$ under $\phi_2$:
\begin{prop}\label{prop:(1,2)}
Let $\alpha \in \calK$. Then $\alpha$ does not realize portrait $(1,2)$ for $f$ if and only if
	\[ \alpha \in \calO_{\phi_2}(0) \cup \calO_{\phi_2}(1/2) \cup \calO_{\phi_2}(\infty). \]
Moreover, for each $k \ge 0$ we have
	\begin{align*}
		h(\phi^k(0)) = h(\phi^k(1/2)) = h(\phi^k(\infty)) = k.
	\end{align*}	
\end{prop}

We begin by giving an alternative description of those points that do not realize portrait $(1,2)$:

\begin{lem}
Let $\alpha \in \calK$. Then $\alpha$ does not realize portrait $(1,2)$ for $f$ if and only if $\alpha = 1/2$ or $\alpha$ satisfies the following conditions:
	\[\tag{$*$} \left\{
	\begin{split}
	&\mbox{$\alpha$ vanishes at $\frakp_{-1}$};\\
	&\mbox{$\Phi_2(-\alpha,t)$ vanishes at $\frakp_{-1}$ and possibly at $\frakp_{-3/4}$, but nowhere else}; and\\
	&\mbox{if $\Phi_2(-\alpha,t)$ vanishes at $\frakp_{-3/4}$, then $\alpha - 1/2$ also vanishes at $\frakp_{-3/4}$}.
	\end{split}
	\right.
	\]
Moreover, if in this case $\Phi_2(-\alpha,t)$ vanishes at $\frakp_{-3/4}$, then it does so to order 1.
\end{lem}

\begin{proof}
By Theorem~\ref{thm:old}, the only constant points $\alpha \in K \subset \calK$ which fail to realize portrait $(1,2)$ are $\alpha = 0$, which satisfies ($*$), and $\alpha = 1/2$. We henceforth assume $\alpha \in \calK \setminus K$, so that $\hhat_f(\alpha) \ge h(\alpha) \ge 1$.

First, suppose that $\alpha$ does not realize portrait $(1,2)$ for $f$. Let $\frakp_c \in \calM_\calK$ be a place at which $\Phi_2(-\alpha,t)$ vanishes. Then $\Phi_1(-\alpha,t)$ or $\alpha$ also vanishes at $\frakp_c$ by Corollary~\ref{cor:fail}. If $\Phi_1(-\alpha,t)$ vanishes at $\frakp_c$, then $\Phi_2(-\alpha(c),c) = \Phi_1(-\alpha(c),c) = 0$. Thus $z = -\alpha(c)$ satisfies both
	\[
		\Phi_2(z,c) = z^2 + z + c + 1 = 0 \mbox{ and } \Phi_1(z,c) = z^2 - z + c = 0,
	\]
which implies that $c = -3/4$ and $\alpha(c) = \alpha(-3/4) = 1/2$. That $t = -3/4$ is a simple root of $\Phi_2(-\alpha,t)$ follows from Remark~\ref{rem:simple}. Since $h(\Phi_2(-\alpha,t)) = 2\hhat_f(\alpha) \ge 2$, and since $\Phi_2(-\alpha,t)$ has at most a simple root at $\frakp_{-3/4}$, $\Phi_2(-\alpha,t)$ must vanish at some place $\frakp_c \ne \frakp_{-3/4}$. In this case, $\Phi_2(-\alpha,t)$ and $\alpha$ must both vanish at $\frakp_c$; hence $0 = \Phi_2(0,c) = c + 1$,
so $ c= -1$. Therefore $\alpha$ satisfies ($*$).

Now suppose $\alpha$ satisfies ($*$). By Corollary~\ref{cor:fail}, it suffices to show that wherever $\Phi_2(-\alpha,t)$ vanishes, so too must $\Phi_1(-\alpha,t)$ or $\alpha$. Since $\alpha$ satisfies ($*$), $\Phi_2(-\alpha,t)$ vanishes only at $\frakp_{-1}$ and possibly $\frakp_{-3/4}$. The result follows by noting that condition ($*$) forces $\alpha$ to vanish at $\frakp_{-1}$, and if $\Phi_2(-\alpha,t)$ vanishes at $\frakp_{-3/4}$, then condition ($*$) says that $\alpha$ reduces to $1/2$ --- and therefore $\Phi_1(-\alpha,t)$ reduces to $\Phi_1(-1/2,-3/4) = 0$ --- modulo $\frakp_{-3/4}$.
\end{proof}

We now verify that $\phi_2$ and $\phi_2^{-1}$ preserve property ($*$), and that $\phi_2$ behaves nicely with respect to the heights of points satisfying ($*$).

\begin{lem}\label{lem:phi(*)}
Let $\alpha \in \calK$ satisfy property ($*$). Then
	\begin{enumerate}
	\item $\phi_2(\alpha)$ satisfies ($*$) as well, and
	\item $h(\phi_2(\alpha)) = h(\alpha) + 1$.
	\end{enumerate}
\end{lem}

\begin{proof}
Since $\alpha$ vanishes at $\frakp_{-1}$, it is clear that $\phi_2(\alpha)$ also vanishes at $\frakp_{-1}$. Also, we have
	\[
		\Phi_2(-\phi_2(\alpha),t) = \phi_2(\alpha)^2 - \phi_2(\alpha) + (t + 1) = \frac{(t + 1)^2(\alpha^2 - \alpha + (t + 1))}{(\alpha + t)^2} = \frac{(t + 1)^2\Phi_2(-\alpha,t)}{(\alpha + t)^2},
	\] 
so $\Phi_2(-\phi_2(\alpha),t)$ vanishes at $\frakp_{-1}$; moreover, since $\Phi_2(-\alpha,t)$ \emph{only} vanishes at $\frakp_{-1}$ and possibly to order one at $\frakp_{-3/4}$, the same holds for $\Phi_2(-\phi_2(\alpha),t)$. (Any pole of $(\alpha + t)^2$ is a pole of at least the same order for $(t+1)^2\Phi_2(-\alpha,t)$, so the above expression may only vanish at the zeroes of its numerator.) Finally, if $\Phi_2(-\phi_2(\alpha),t)$ vanishes at $\frakp_{-3/4}$, then so must $\Phi_2(-\alpha,t)$, in which case $v_{-3/4}(\alpha - 1/2) > 0$. Hence
	\[
	\phi_2(\alpha) - \frac{1}{2} = \frac{(t + 1/2)(\alpha - 1/2) - (t + 3/4)}{(\alpha - 1/2) + (t + 1/2)}
	\]
vanishes at $\frakp_{-3/4}$. Therefore $\phi_2(\alpha)$ satisfies ($*$).

We now prove (B). Letting
	\[
	\gamma := \frac{1}{\phi_2(\alpha)} = \frac{1}{t + 1} + \frac{1}{\alpha - 1},
	\]
it suffices to show that $h(\gamma) = h(\alpha) + 1$. Suppose that $\frakp$ is a pole of $\gamma$. Then either $\frakp = \frakp_{-1}$, in which case property ($*$) implies that $v_\frakp(\alpha - 1) = 0$, hence $v_\frakp(\gamma) = -v_\frakp(t + 1) = -1$; or else $\frakp \ne \frakp_{-1}$ is a zero of $\alpha - 1$, in which case $v_\frakp(\gamma) = -v_\frakp(\alpha - 1)$. Therefore
	\[
	h(\gamma)
		= -\sum_{v_\frakp(\gamma) < 0} v_\frakp(\gamma)
		= - v_{-1}(\gamma) + \sum_{v_\frakp(\alpha - 1) > 0} v_\frakp(\alpha - 1)
		= 1 + h(\alpha - 1).
	\]
Since $h(\alpha - 1) = h(\alpha)$, we are done.
\end{proof}

\begin{lem}\label{lem:phi-1(*)}
Let $\alpha \in \calK$ satisfy ($*$), and assume $\alpha \not \in \{0, t + 1, -\frac{t + 1}{2t + 1}\}$.
	\begin{enumerate}
	\item We have
		\begin{align*}
		v_{-1}(\alpha) &= 1,\\
		v_{-1}(\alpha - (t + 1)) &= 2, \mbox{ and } \\
		v_{-1}(t\alpha + (t + 1)) &\ge 3.
		\end{align*}
	\item The point $\phi_2^{-1}(\alpha)$ also satisfies ($*$).
	\end{enumerate}
\end{lem}

\begin{proof}
We first show that $v_{-1}(\Phi_2(-\alpha,t)) \ge 3$. Since $\Phi_2(-\alpha,t)$ can only vanish at $\frakp_{-1}$ and possibly, to order one, at $\frakp_{-3/4}$, we have
	\begin{align*}
	v_{-1}(\Phi_2(-\alpha,t)) &= h(\Phi_2(-\alpha,t)) - v_{-3/4}(\Phi_2(-\alpha,t))\\
		&= 2\hhat_f(\alpha) - 
		\begin{cases}
		0, &\mbox{ if $\Phi_2(-\alpha,t)$ does not vanish at $\frakp_{-3/4}$},\\
		1, &\mbox{ if $\Phi_2(-\alpha,t)$ vanishes at $\frakp_{-3/4}$}.
		\end{cases}
	\end{align*}
Thus $v_{-1}(\Phi_2(-\alpha,t)) \le 2$ if and only if $\hhat_f(\alpha) \le 1$ or $\hhat_f(\alpha) = 3/2$ and $\Phi_2(-\alpha,t)$ vanishes at $\frakp_{-3/4}$. A simple calculation then verifies that the only such $\alpha \in \calK$ satisfying ($*$) are the three values of $\alpha$ excluded from the statement of the lemma. Therefore $v_{-1}(\Phi_2(-\alpha,t)) \ge 3$.

Now, write
	\[
	\alpha = (t + 1) + \alpha^2 - (\alpha^2 - \alpha + (t + 1)) = (t + 1) + \alpha^2 - \Phi_2(-\alpha,t).
	\]
Since $v_{-1}(\alpha) \ge 1$ by assumption, we have $v_{-1}(\alpha^2) \ge 2$, thus $v_{-1}(\alpha) = v_{-1}(t+1) = 1$. This implies that $v_{-1}(\alpha^2) = 2$, and therefore $v_{-1}(\alpha - (t + 1)) = v_{-1}(\alpha^2 - \Phi_2(-\alpha,t)) = 2$ as well. Finally, we write
	\[
	t\alpha + (t + 1) = \Phi_2(-\alpha,t) - \alpha(\alpha - (t + 1)),
	\]
and by what we have already shown, this has valuation at least 3 at $\frakp_{-1}$, proving (A).

We now show that $\phi_2^{-1}(\alpha)$ satisfies ($*$). By part (A), we have
	\[
		v_{-1}(\phi_2^{-1}(\alpha)) = v_{-1}(t\alpha + (t + 1)) - v_{-1}(\alpha - (t + 1)) \ge 1,
	\]
so $\phi_2^{-1}(\alpha)$ vanishes at $\frakp_{-1}$. Now consider
	\[
		\Phi_2(-\phi_2^{-1}(\alpha),t)
			= \left(-\phi_2^{-1}(\alpha)\right)^2 - \phi_2^{-1}(\alpha) + t + 1
			= \frac{(t+1)^2\Phi_2(-\alpha,t)}{(\alpha - (t+1))^2}.
	\]
Since $\Phi_2(-\alpha,t)$ vanishes to order at least three and $(\alpha - (t+1))^2$ vanishes to order four at $\frakp_{-1}$, it follows that $\Phi_2(-\phi_2^{-1}(\alpha),t)$ vanishes at $\frakp_{-1}$. Moreover, $\Phi_2(-\phi_2^{-1}(\alpha),t)$ can only vanish at $\frakp_{-1}$ and the places at which $\Phi_2(-\alpha,t)$ vanishes, which are only $\frakp_{-1}$ and possibly $\frakp_{-3/4}$ by assumption. (As before, we are using the fact that the above expression cannot vanish at poles of its denominator.)

Finally, suppose $\Phi_2(-\phi_2^{-1}(\alpha),t)$ vanishes at $\frakp_{-3/4}$. This is equivalent to the vanishing of $\Phi_2(-\alpha,t)$, necessarily to order one, at $\frakp_{-3/4}$, in which case $\alpha - 1/2$ also vanishes at $\frakp_{-3/4}$. Thus
	\[
		v_{-3/4}(\Phi_2(-\phi_2^{-1}(\alpha),t)) = v_{-3/4}(\Phi_2(-\alpha,t)) = 1.
	\]
Moreover, we have
	\[
	\phi_2^{-1}(\alpha) - \frac{1}{2} = -\frac{(t + 1/2)(\alpha - 1/2) + (t + 3/4)}{\alpha - (t + 1)},
	\]
which vanishes at $\frakp_{-3/4}$ by our assumptions on $\alpha$. Therefore $\phi_2^{-1}(\alpha)$ satisfies ($*$) as well.
\end{proof}

\begin{proof}[Proof of Proposition~\ref{prop:(1,2)}]
Suppose first that $\alpha$ fails to realize portrait $(1,2)$ for $f$. Since clearly $1/2 \in \calO_{\phi_2}(1/2)$, we will assume that $\alpha \ne 1/2$, so $\alpha$ satisfies ($*$). We proceed by induction on $h(\alpha)$.

By Theorem~\ref{thm:old}, the only constant points $\alpha \in K \subset \calK$ that do not realize portrait $(1,2)$ for $f$ are $\alpha \in \{0,1/2\}$; thus the $h(\alpha) = 0$ case holds. Now suppose $h(\alpha) \ge 1$. Since $t + 1 = \phi_2(\infty)$ and $-(t+1)/(2t+1) = \phi_2(1/2)$ we will assume $\alpha \not \in \{t + 1, -(t+1)/(2t+1)\}$. Then Lemma~\ref{lem:phi-1(*)} says that $\phi_2^{-1}(\alpha)$ satisfies ($*$), hence $h(\phi_2^{-1}(\alpha)) = h(\alpha) - 1$ by Lemma~\ref{lem:phi(*)}. By induction, we have $\phi_2^{-1}(\alpha) \in \calO_{\phi_2}(\delta)$ for some $\delta \in \{0,1/2,\infty\}$, and therefore $\alpha \in \calO_{\phi_2}(\delta)$ as well.

Now suppose $\alpha \in \calO_{\phi_2}(\delta)$ for some $\delta \in \{0,1/2,\infty\}$. Write $\alpha = \phi_2^k(\delta)$ for some $k \ge 0$. We show by induction on $k$ that $h(\alpha) = k$ and that $\alpha$ satisfies ($*$), which is equivalent (for $\alpha \ne 1/2$) to the assertion that $\alpha$ fails to realize portrait $(1,2)$ for $f$.

Since 0 and $1/2$ fail to realize portrait $(1,2)$ for $f$, the statement holds for $k = 0$. The points $\phi_2(0) = -(t+1)/t$, $\phi_2(1/2) = -(t+1)/(2t + 1)$, and $\phi_2(\infty) = t + 1$ all satisfy ($*$), and certainly all three have height equal to one, establishing the $k = 1$ case. Now suppose $k \ge 2$. By induction, $\phi_2^{k-1}(\delta)$ satisfies property ($*$) and has height $k - 1$. By Lemma~\ref{lem:phi(*)}, we conclude that $\alpha = \phi_2^k(\delta)$ satisfies ($*$) and has height $k$, completing the proof.
\end{proof}

In order to prove the more general statement involving points of portrait $(M,2)$ with $M \ge 2$, we require the following:

\begin{lem}\label{lem:(1,2)poles}
Let $k \ge 1$. Then
	\begin{align*}
	v_\infty(\phi_2^k(0)) = v_\infty(\phi_2^k(1/2)) &= 0; \text{ and }\\
	v_\infty(k\phi_2^k(\infty) - t) &\ge 0.
	\end{align*}
\end{lem}

\begin{proof}
The map $\phi_2$ reduces to $z - 1$ modulo $\frakp_\infty$, so $\phi_2^k(\delta) \equiv \delta - k$ (mod $\frakp_\infty$) for all $\delta \in K$ and $k \in \bbN$. If we take $\delta \in \{0,1/2\}$, then $\delta - k$ is a nonzero constant for all $k \ge 1$, thus $v_\infty(\phi_2^k(\delta)) = 0$.

Now, for each $k \ge 1$ we set $u_k := k\phi_2^k(\infty) - t$. We show that $v_\infty(u_k) \ge 0$ by induction on $k$. The result clearly holds for $k = 1$, since $u_1 = \phi_2(\infty) - t = 1$, so we assume $k \ge 2$. Then
	\begin{align*}
	u_k = k\phi_2(\phi_2^{k-1}(\infty)) - t &= k\phi_2\left(\frac{u_{k-1} + t}{k - 1}\right) - t\\
		&= \frac{t((k - 1)u_{k-1} - k^2 + 2k) + (ku_{k-1} - k^2 + k)}{u_{k-1} + kt}.
	\end{align*}
By induction, we have $v_\infty(u_{k-1}) \ge 0$, and it therefore follows that $v_\infty(u_k) \ge 0$.
\end{proof}

We now prove Theorem~\ref{thm:main} in the case $d = N = 2$.

\begin{prop}\label{prop:maind2N2}
Let $\alpha \in \calK$, and let $M \ge 0$. Then $\alpha$ realizes portrait $(M,2)$ for $f_2$ if and only if $(\alpha,M)$ does not satisfy one of the following conditions:
	\begin{itemize}
	\item $M = 0$ and $\alpha = -1/2$;
	\item $M = 1$ and $\alpha \in \calO_{\phi_2}(0) \cup \calO_{\phi_2}(1/2) \cup \calO_{\phi_2}(\infty)$; or
	\item $M = 2$ and $\alpha = \pm 1$.
	\end{itemize}
Moreover, for each $k \ge 0$, we have $h(\phi_2^k(0)) = h(\phi_2^k(1/2)) = h(\phi_2^k(\infty)) = k$.
\end{prop}

\begin{proof}
The $M = 0$ and $M = 1$ cases follow from Propositions~\ref{prop:periodic} and \ref{prop:(1,2)}, respectively. We therefore assume $M \ge 2$.

That $\pm 1$ do not realize portrait $(2,2)$ for $f$ follows from Theorem~\ref{thm:old}. Now suppose that $\alpha$ does not realize portrait $(M,2)$ for $f$; equivalently, suppose that $f^{M-1}(\alpha)$ does not realize portrait $(1,2)$ for $f$. By Proposition~\ref{prop:(1,2)}, we have $f^{M-1}(\alpha) = \phi_2^k(\delta)$ for some $k \ge 0$ and $\delta \in \{0,1/2,\infty\}$. Lemma~\ref{lem:hhat_ineq} asserts that any point in the image of $f$ must have a pole at $\frakp_\infty$; since points in the orbits of 0 and $1/2$ under $\phi_2$ do not have poles at $\frakp_\infty$ by Lemma~\ref{lem:(1,2)poles}, we must have $f^{M-1}(\alpha) = \phi_2^k(\infty)$ for some $k \ge 0$. Since the only preimage of $\infty$ is $\infty$ itself, we must have $k \ge 1$.

Set $\beta := f^{M-2}(\alpha)$. Then $f(\beta) = \phi_2^k(\infty)$, so by Lemma~\ref{lem:(1,2)poles} we have that
	\[
		k\phi_2^k(\infty) - t = kf(\beta) - t = k\beta^2 + (k - 1)t
	\]
is regular at $\frakp_\infty$. This implies that $k = 1$, so $f^{M-1}(\alpha) = \phi_2(\infty) = t + 1$. The preimages of $t + 1$ under $f$ are $\pm 1$, and the preimages of $\pm 1$ lie outside of $\calK$. Therefore, if $f^{M-1}(\alpha) = t + 1$ for some $M \ge 2$ and $\alpha \in \calK$, then $M = 2$ and $\alpha = \pm 1$, as claimed.
\end{proof}

As mentioned at the beginning of this section, the proof of the following statement --- the $d = 2$, $N = 1$ case of the main theorem --- uses the same ideas as for Proposition~\ref{prop:maind2N2}, and we therefore omit the proof.

\begin{prop}\label{prop:maind2N1}
Let $\alpha \in \calK$, and let $M \ge 0$. Then $\alpha$ fails to realize portrait $(M,1)$ for $f_2$ if and only if $M = 1$ and $\alpha \in \calO_{\phi_1}(0) \cup \calO_{\phi_1}(\infty)$, where
	\[
	\phi_1(z) = -\frac{t(z + 1)}{z - (t - 1)}.
	\]
Moreover, for each $k \ge 0$, we have $h(\phi_1^k(0)) = h(\phi_1^k(\infty)) = k$.
\end{prop}
Proposition~\ref{prop:maind2N1} is the final case of Theorem~\ref{thm:main}. Therefore, by combining Propositions~\ref{prop:maind2N2} and \ref{prop:maind2N1} with the results of the previous sections, we have proven the main theorem.

\bibliography{C:/Dropbox/jdoyle}

\bibliographystyle{amsplain}

\end{document}